\documentclass[11pt]{article}
\usepackage{amssymb}
\usepackage{graphicx}
\usepackage{xcolor} 
\usepackage{tensor}
\usepackage{fullpage} 
\usepackage{amsmath}
\usepackage{amsthm}
\usepackage{mathrsfs}
\usepackage{verbatim}
\usepackage{hyperref}
\usepackage{cite}
\usepackage{array} 
\usepackage{bbm}
\usepackage{url}

\usepackage{enumitem}
\setlist[enumerate]{leftmargin=1.2em}
\setlist[itemize]{leftmargin=1.2em}
\usepackage{seqsplit,mathtools}

\newtheorem{theorem}{Theorem}[section]
\newtheorem{corollary}[theorem]{Corollary}
\newtheorem{lemma}[theorem]{Lemma}
\newtheorem{conjecture}[theorem]{Conjecture}
\newtheorem{proposition}[theorem]{Proposition}

\theoremstyle{definition}
\newtheorem{definition}[theorem]{Definition}

\theoremstyle{remark}
\newtheorem{remark}[theorem]{Remark}

\numberwithin{equation}{section}

\title{Small scale creation of the Lagrangian flow in 2d perfect fluids}
\author{Ayman Rimah Said}
\date{\ }

\begin{document}

\maketitle
\begin{abstract}
   In this paper we prove that for all solutions of the 2d Euler equations with initial vorticity with finite Sobolev smoothness an initial data dependent norm  of the associated Lagrangian flow blows up in infinite time at least like $t^{\frac{1}{3}}$. This initial data dependent norm quantifies the exact $L^2$ decay of the Fourier transform of the solution. This adapted norm turns out to be the exact quantity that controls a low to high frequency cascade which we then show to be the quantitative phenomenon behind the Lyapunov construction by Shnirelman in \cite{ShnLyap}.
\end{abstract}
\section{Introduction}
We study 2d inviscid flows
\begin{equation}
    \label{eq: 2dEuler1} \partial_t \omega + u\cdot\nabla \omega=0,
\end{equation}
\begin{equation}\label{eq: 2dEuler2}
    u=\nabla^\perp \psi \text{ and }\Delta\psi=\omega.
\end{equation} 
Here, the scalar vorticity $\omega:\mathbb{R}^2\times \mathbb{R}\rightarrow\mathbb{R}$ is transported by the velocity field $u:\mathbb{R}^2\times\mathbb{R}\rightarrow\mathbb{R}^2$ which is uniquely determined at each time $t\in\mathbb{R}$ from $\omega$ using the Newtonian potential:
\begin{equation}\label{Newtonian}u(x)=\frac{1}{2\pi}\int_{\mathbb{R}^2}\frac{(x-y)^\perp}{|x-y|^2}\omega(y)dy, \text{ and reciprocally } \omega=\nabla \times u=\partial_1 u_2-\partial_2 u_1.\end{equation} We adopt the standard notation $v^\perp=(-v_2, v_1)$ for $v=(v_1,v_2)\in\mathbb{R}^2.$ It is well known that smooth enough solutions to the 2d Euler equations \eqref{eq: 2dEuler1}-\eqref{eq: 2dEuler2} retain their smoothness for all finite times. Much less is known in the infinite-time limit. The long time behavior seems to consistently show some type of small scale creation for smooth solutions\cite{drivas2022singularity,khesin2023geometric}, which can be summed up in the following conjecture by Yudovich.
\begin{conjecture}[Yudovich (1974), \cite{yudovic1974loss,yudovich2000loss}, quote from \cite{morgulis2008loss}]\label{conj:generic small scale creation}
 There is a ``substantial set" of inviscid incompressible flows whose vorticity gradients grow without bound. At least this set is dense enough to provide the loss of smoothness for some arbitrarily small disturbance of every steady flow.  
\end{conjecture}
The literature towards this conjecture is rich. Of note is the result of Koch \cite{koch2002transport} in which strong growth of H\"older and Sobolev norms of the vorticity is established near any background solution (stationary or time-dependent) for which the gradient of the flow map is unbounded in time. Yudovich also established (boundary induced) growth results under some mild assumption on the data near the boundary of the domain \cite{yudovich2000loss} (see also \cite{morgulis2008loss} for an extension of \cite{yudovich2000loss}). The conjecture was established within m-fold symmetry for $m\geq 3$ by Elgindi, Murray and the author in \cite{EMS}. There are also numerous important results on growth of solutions in the neighborhood of stable steady states \cite{drivas2022singularity,denisov2009infinite,nadirashvili1991wandering,kiselev2014small,zlatovs2015exponential}. In the case of open neighborhoods of shearing stable steady states a finer version of the conjecture including generic fluid aging has been recently established by Drivas, Elgindi and Jeong \cite{drivas2023twisting}.

The main Theorem of this paper can be stated informally as follows.
\begin{theorem}\label{thm:informal} Consider the 2d Euler equation on $\mathbb{R}^2.$ Then for $\omega_0\in H^{s}(\mathbb{R}^2)\setminus H^{s+\epsilon}(\mathbb{R}^2)$ for some $s>1$ and all $\epsilon>0$, then roughly the s+1 derivative of the Lagrangian flow blows up in infinite time at least like $t^{\frac{1}{3}}$. \end{theorem}
All the results in this paper can be naturally generalised to the periodic 2d Euler equation on $\mathbb{T}^2$ and an analogue of Theorem \ref{thm:informal} holds on $\mathbb{T}^2$. 

\subsection{Propagation of exact smoothness in the 2d Euler equation}

We recall the standard global well-posedness theory of the 2d Euler equation in sub-critical Sobolev spaces, see for example the following excellent monographs as an introduction to the study of the Euler equations \cite{majda2002vorticity,chemin1995fluides, gerard1992resultats, marchioro2012mathematical}. We define the Lagrangian flow
\[
\frac{d}{dt}\Phi_t=u\circ \Phi_t \text{ with } \Phi_0(\cdot)=Id.
\] 
\begin{theorem}\label{thm:welp 2d eul}
    Consider $s>1$ and $\omega_0\in H^{s}\left( \mathbb{R}^2\right)$, then there exists a unique solution $\omega\in C\left(\mathbb{R},H^{s}\left( \mathbb{R}^2\right)\right)$ of \eqref{eq: 2dEuler1}-\eqref{eq: 2dEuler2} with initial data $\omega_0$. Moreover  there exists a universal constant $C_s$ such that we have following estimate
\[
\left\Vert \omega(t) \right\Vert_{H^{s}}+\left\Vert \Phi_t-Id \right\Vert_{H^{s+1}}\leq C_{s} \left\Vert \omega_0 \right\Vert_{H^s}\exp\left(\exp\left(C_{s} \left\Vert \omega_0 \right\Vert_{L^2\cap L^\infty} t\right)\right).
\]
\end{theorem}

One of the first results of this paper is to describe a \textit{maximal} propagation of smoothness result for the 2d Euler equation. First we set, for $\omega_0\in L^2$, $\epsilon \geq 0$
\[
dr_{\omega_0}(\epsilon)=\left(\int_{\mathbb{R}^2\setminus B\left(0,\frac{1}{\epsilon}\right)}\left|\mathscr{F}(\omega_0)(\xi)\right|^2 d\xi\right)^{\frac{1}{2}}
\]
 We note that $dr_{\omega_0}$ is an increasing function of $\epsilon$ with $\displaystyle \lim_{\epsilon\to 0}  dr_{\omega_0}(\epsilon)=0$ and $\displaystyle \lim_{\epsilon\to +\infty}  dr_{\omega_0}(\epsilon)=\left\Vert \omega_0\right\Vert_{L^2(\mathbb{R}^2)}$. For $f\in \mathscr{S}(\mathbb{R}^2)$ such that $dr_{f}$ vanishes at $0$ at least like $dr_{\omega_0}$, we define
\[
\left\Vert f \right\Vert_{\omega_0}=\sup_{\epsilon\geq 0} \frac{dr_{f}(\epsilon)}{dr_{\omega_0}(\epsilon)}.
\]
Note that by construction $\left\Vert \omega_0 \right\Vert_{\omega_0}=1$.
\begin{theorem}\label{thm: specific smooth prop}
    Consider $s>1$ and $\omega_0\in H^{s}\left( \mathbb{R}^2\right)$ and $\omega\in C\left(\mathbb{R},H^{s}\left( \mathbb{R}^2\right)\right)$ the unique solution of \eqref{eq: 2dEuler1}-\eqref{eq: 2dEuler2} with initial data $\omega_0$. Suppose that for all $\lambda<1$
    \begin{equation}\label{eq:decay fourier hyp}
    C_{\omega_0}(\lambda) dr_{\omega_0}( \epsilon)\leq dr_{\omega_0}(\lambda \epsilon),
    \end{equation}
    for some function $C_{\omega_0}$ independent of $\epsilon$. Then there exists a constant $C_{\omega_0}'$ such that 
    \[
    \left\Vert \omega(t) \right\Vert_{\omega_0}+\left\Vert D\Phi_t -Id\right\Vert_{\omega_0}\leq  \exp\left(\exp\left(C_{\omega_0}' t\right)\right)
    \]
\end{theorem}
 From the proof of the of the previous theorem we show that the condition \eqref{eq:decay fourier hyp} is equivalent to the fact that $\omega_0$'s Fourier transform decays at most algebraically fast at infinity, in particular the previous theorem applies for all $\omega_0\in H^{s_1}(\mathbb{R}^2)\setminus H^{s_2}(\mathbb{R}^2)$ for a pair $1<s_1<s_2$. We believe that this type of exact smoothness propagation holds more generally for hyperbolic evolution PDEs. For example the proof here works in verbatim to give an analogous result (locally in time) for the SQG equation. 
 
 In some sense this is an optimal sub-critical smoothness propagation statement and answers the $L^2$ based version of Problem 23 of \cite{khesin2023geometric} for the persistence of ``kinks" in fluid flows. It is not hard to see from the proof of Theorem \ref{thm: specific smooth prop}, specifically from Proposition \ref{prop: charac norm omeg} that an analogous definition of an $\omega_0$ dependent norm that captures sub-critical $L^p$ based ``kinks"  are also being propagated by the flow of the 2d Euler equations for $p\in [1,+\infty]$.

\subsection{The forward frequency cascade}
Towards understanding long time behavior of a dynamical system one of the main tools is the construction of Lyapunov functions, in the case of the 2d Euler equations very few such examples are known \cite{EMS,JS,morgulis2008loss,ShnLyap,yudovich2000loss}. The proof of Theorem \ref{thm:informal} relies on the careful study of Shnirelman's pioneering Lyapunov construction \cite{ShnLyap} and supplementing it with optimal quantitative estimates which turn out to be the construction's most natural setting and in particular for the so called ``microlocal scalar product". This allows us to significantly simplify the construction. \begin{remark}
    We record here part of S. Alinhac's MathSciNet review of A. Shnirelman's paper: ``  This may all seem very complicated, but it is not, the technical ``complications'' arising naturally in the course of a basically simple argument. Moreover, the new tool (the microlocal scalar product) introduced by the author is certainly likely to have many other applications, just as similar tools (microlocal defect measures, Young measures, etc.) already have in the theory of weak solutions of nonlinear PDE, in homogenization, in control theory, etc. Finally, we would like to emphasise what we believe can be learned from the approach of the author: considering such an old problem as the fluid flow, it is likely that not many new mathematical results are going to be obtained by nineteenth-century PDE methods. Even researchers oriented towards applications will have to incorporate Shnirelman's results into their research, just because these results, far from being some (irrelevant) refinement of basically well understood things, are the first rigorous ones in the subject."
\end{remark}
To state the theorem on the Lyapunov construction we need to introduce the notion of paraproducts. For the reader unfamiliar with paraproducts we give a heuristic construction below.  \subsubsection*{\textbf{$\bullet$ Paraproducts}} 
For the sake of this discussion let us pretend that $\partial_x$ is left-invertible with a choice of $\partial_x^{-1}$ that acts continuously from $H^s$ to $H^{s+1}$. We follow here analogous ideas to the ones presented by Shnirelman in \cite{shnirelman2005microglobal}. One way to define the paraproduct of two functions $f,g\in H^s$ with $s$ sufficiently large is: we differentiate $fg$ $k$ times, using the Leibniz formula, and then restore the function $fg$ by the $k$-th power of $\partial_x^{-1}$:
 \begin{align*}
 fg&=\partial_x^{-k}\partial_x^{k}(fg)\\
 	&=\partial_x^{-k}\big(g\partial_x^k f+k\partial_x g\partial_x^{k-1} f+\dots+k\partial_x f\partial_x^{k-1} g+g\partial_x^k f  \big)\\
 	&=T_g f+T_fg+R,
 \end{align*}
 where,
 \[T_gf=\partial_x^{-k}\big(g\partial_x^k f\big), \ \ T_fg=\partial_x^{-k}\big(f\partial_x^k g\big),\]
 and $R$ is the sum of all remaining terms. The key observation is that if $s>\frac{1}{2}+k$, then $g \mapsto T_fg$ is a continuous operator in $H^s$ for $f  \in H^{s-k}$. The remainder $R$ is a continuous bilinear operator from $H^s$ to $H^{s+1}$. The operator $T_fg$ is called the paraproduct of $g$ and $f$ and can be interpreted as follows. The term $T_fg$ takes into play high frequencies of $g$ compared to those of $f$ and demands more regularity in $g\in H^s$ than $f \in H^{s-k}$ thus the term $T_fg$ bears the ``singularities" brought on by $g$ in the product $fg$. Symmetrically $T_gf$ bears the "singularities" brought on by $f$ in the product $fg$ and the remainder $R$ is a smoother function ($H^{s+1}$) and does not contribute to the main singularities of the product. 
 \subsubsection*{\textbf{$\bullet$ Paradifferential operators}} 
 To get a good intuition of a paradifferential operator $T_{p(x,\xi)}$ with symbol $p(x,\xi)\in \Gamma^\beta_\rho(\mathbb{R}^2)$, as a first gross approximation, one can think of $p(x,\xi)\approx f(x)m(\xi)$ and $T_{p(x,\xi)}$ as the composition of a paraproduct $T_f$ with a Fourier multiplier $m(D)$, that is:
 \[
 T_{p(x,\xi)}\approx T_f m(D), \text{ with } f\in W^{\rho,\infty} \text{ and }m \text{ is of order }\beta.
 \]
 Indeed following Coifman and Meyer's symbol reduction given in Proposition $5$ of \cite{coifmandela}, one can show that linear combinations of composition of a paraproduct with a Fourier multiplier are dense in the space of paradifferential operators. Definition \ref{def: para op} gives the rigorous formula for $T_{p(x,\xi)}$.
 
We are now in position to state the theorem summarising the Lyapunov construction. 
\begin{theorem}\label{thm: intro lyap}
    Consider $\chi(\xi)\in C^\infty_0( \mathbb{R}^2\setminus B(0,1))$ and $\omega_0\in H^{s}$ with $s>1$ verifying \eqref{eq:decay fourier hyp} then there exists a universal constant $C$ and a constant $C_{\omega_0}$ such that for $\epsilon \geq 0$
\begin{align*}
    \frac{d}{dt}\left(\nabla \times T_{[D\Phi_t]^{-1}}\Phi_t,\chi(\epsilon D)\omega_0\right)_{L^2}=\left\Vert T_{\frac{\left\vert\xi\right\vert}{\left\vert [D\Phi_t]^{-1}\xi\right\vert}}\chi(\epsilon D)\omega_0\right\Vert^2_{L^2}+ O\left(C_\delta e^{Ce^{C_{\omega_0}t}}\epsilon^{\min(s-1-\delta,1)}  dr_{\omega_0}(\epsilon)^2\right),
    \end{align*}
for all $0<\delta<s-1$ and $C_\delta>0$ is constant depending only on $\delta$.
\end{theorem}
\begin{remark}
   In \cite{ShnLyap} Shnirelman proved the previous theorem in the case $dr_{\omega_0}(\epsilon)=O(\epsilon^{s}), s>2$.  Theorem \ref{thm: intro lyap} generalises this result to the exact regularity of $\omega_0$ whatever it is and gives the optimal control in $\epsilon$.
\end{remark}
In particular the leading order decay in $\epsilon $ of $\left(\nabla \times T_{[D\Phi_t]^{-1}}\Phi_t,\chi(\epsilon D)\omega_0\right)$ is
\[
\int_0^t\left\Vert T_{\frac{\left\vert\xi\right\vert}{\left\vert [D\Phi_r]^{-1}\xi\right\vert}}\chi(\epsilon D)\omega_0\right\Vert^2_{L^2}dr\underset{\epsilon\to 0}{\sim} c(t)dr_{\omega_0}(\epsilon)^2
\]
where $c(t)$ is an increasing function of time. The explicit estimate on the residual term allows for the following interpretation of the previous result. Fixing an outer frequency region $\{\left\vert\xi\right \vert\geq R \}$ then there exists $T_{R}>0$ increasing in $R$ such that for $\left\vert t\right\vert \leq T_{R}$ there is an averaged forward frequency cascade of $\Phi_t$ in the signed measure $\mathscr{F}(\omega_0)(\xi)d\xi$ into the region $\{\left\vert\xi\right \vert\geq R \}$. Thus there is always a positive flux of frequency at ``infinity'' ($R\to \infty$) and the growth of that rate gives the desired Lyapunov function. Using $dr_{\omega_0}$, it is given explicitly by a re-normalised version of the semi-classical measure first introduced in \cite{gerard1991mesures} and independently as the Wigner measure in \cite{lions1993mesures} which is the so called microlocal scalar product introduced by Shnirelman in \cite{ShnLyap}
\[
L_{\chi,\omega_0}(\Phi_t)=\limsup_{\epsilon \to 0}\frac{\left(\chi(\epsilon D)\nabla \times T_{[D\Phi_t]^{-1}}\Phi_t,\chi(\epsilon D)\omega_0\right)_{L^2}}{dr_{\omega_0}(\epsilon)^2}.
\]
To give a geometrical interpretation of what is measured by this Lyapunov function we need to heuristically introduce the notion of paracomposition which is defined rigorously in Theorem \ref{thm: def para comp}.
 \subsubsection*{\textbf{$\bullet$ Paracomposition}} 
We again work with $f \in H^s$ and $g \in C^s$ with $s$ large and consider the composition of two functions $f\circ g$ which bears the singularities of both $f$ and $g$, and our goal is to separate them. We proceed as before by differentiating $f \circ g$ $k$ times, using the Fa\'a di Bruno's formula, and then restore the function $fg$ by the $k$-th power of $\partial_x^{-1}$:
\begin{align*}
 f \circ g&=\partial_x^{-k} \partial_x^{k} (f \circ g)\\
 	&=\partial_x^{-k}\big((\partial_x^k f\circ g)\cdot(\partial_x g)^k 
 	+\dots+(\partial_x f\circ g)\cdot\partial_x^k g  \big)\\
 	&=g^*f+T_{\partial_x f \circ g}g+R,
 \end{align*}
 where,
 \[g^*f=\partial_x^{-k}\big((\partial_x^k f\circ g)\cdot(\partial_x g)^k \big) \text{ is the paracomposition of $f$ by $g$}\]
 and $R$ is the sum of all remaining terms. Again the key observation is that if $s>\frac{1}{2}+k$, then $f \mapsto g^*f$ is a continuous operator in $H^s$ for $g  \in C^{s-k}$. Thus this term bears essentially the singularities of $f$ in $f\circ g$. As before $T_{\partial_x f \circ g}g$ bears essentially the singularities of $g$ in $f\circ g$. The remainder $R$ is a continuous bilinear operator from $H^s$ to $H^{s+1}$. Thus we have separated the singularities of the composition $f\circ g$.

The geometric interpretation of the left hand side in Theorem \ref{thm: intro lyap} is computed explicitly in the proof as
\[
\frac{d}{dt} T_{[D\Phi_t]^{-1}}\left(\Phi_t-Id\right)\approx T_{[D\Phi_t]^{-1}}\Phi_t^* u,
\]
 now recall that the pull back of $u$ by $\Phi_t$ is given by $[D\Phi_t]^{-1} u\circ \Phi_t$, then the right hand side in the previous identity can be interpreted as a paradifferential version of this pull-back which ``selects" the high frequencies of $u$ compared to $\Phi_t$. Thus 
 \[
 \frac{d}{dt} \left(\nabla \times T_{[D\Phi_t]^{-1}}\left(\Phi_t-Id\right),\chi(\epsilon D)\omega_0\right)_{L^2}\approx \left(\nabla \times T_{[D\Phi_t]^{-1}}\Phi_t^* u,\chi(\epsilon D)\omega_0\right)_{L^2},
 \]
and the key observation is that the projection of the curl of the ``para"pull-back of the velocity to the high frequencies  of the initial vorticity is positive to leading order.

Another interpretation of the dynamical phenomena captured in the previous theorem which is essentially the idea behind the proof and the key observation in \cite{ShnLyap} is the following. We write 
\[
\frac{d}{dt}\Phi_t=[\nabla^{\perp}\Delta^{-1}\left(\omega_0\circ \Phi_t^{-1} \right)]\circ \Phi_t=\left(\nabla^{\perp}\Delta^{-1}\right)^{*,\Phi_t}\omega_0,
\]
where $\left(\nabla^{\perp}\Delta^{-1}\right)^{*,\Phi_t}$ is the pulled back operator by $\Phi_t$. Thus morally the main term of the left hand is $\nabla \times T_{[D\Phi_t]^{-1}}\left(\nabla^{\perp}\Delta^{-1}\right)^{*,\Phi_t}$. Now note that initially 
\[
\nabla \times T_{[D\Phi_t]^{-1}}\left(\nabla^{\perp}\Delta^{-1}\right)^{*,\Phi_t}_{|t=0}=\nabla \times T_{Id}\nabla^{\perp}\Delta^{-1}\approx Id,
\]
thus the identity given in Theorem \ref{thm: intro lyap} can be interpreted as a measure of how the flow ``twists" the identity. More precisely the use of the semi-classical cut-off in frequency, $\chi(\epsilon D)$, gives that the left hand-side of the identity given in Theorem \ref{thm: intro lyap} is the study of the principal symbol of the operator $\nabla \times T_{[D\Phi_t]^{-1}}\left(\nabla^{\perp}\Delta^{-1}\right)^{*,\Phi_t}$. The right hand side then gives that in the small scales ($\epsilon \to 0$) the Lagrangian flow $\Phi_t$ is consistently getting farther and farther from the identity, that is small scales are continuously being created.

The use of $T_{[D\Phi_t]^{-1}}$ in the previous operator is not only the natural factor appearing in the pull-back of a velocity field $u$ by a flow $\Phi_t$, it also concretly brings into play a key cancellation in the principal symbol that can be seen as follows:
\[
\frac{d}{dt}\left([D\Phi_t]^{-1}\Phi_t\right)=[D\Phi_t]^{-1}\left(u\circ \Phi_t-Du\circ\Phi_t \Phi_t \right).
\]
The right hand side is of the form $G(x)=F(x)-F'(x)x$ and we note the cancellation $G'(x)=-F''(x)x$. To get useful quantitative bounds from this cancellation in the high frequencies it is natural to use paraproducts and paracomposition. Then the formula obtained on the right hand side in Theorem \ref{thm: intro lyap} can be seen as a direct consequence of this cancellation combined with the algebra property of pseudodifferential and paradifferential operators combined with the fact that restricted to principal symbols this algebra becomes commutative.

\begin{remark}
The proof of the previous theorem applies more generally for $\chi(\epsilon D)$ replaced with $a(x,\epsilon D)$ with $a$ in the (standard) H\"ormander symbol class $S^0_{1,0}(\mathbb{R}^2)$ and $T_a$ with $a$ in the paradifferential symbol class $\Gamma_0^1(\mathbb{R}^2)$. The control on $a$ needed in the residual term is it's symbolic semi-norm $M_1^0(a;2)$ given in Definition \ref{def: pseudo symb}. Note that a large number of such Lyapunov functions can be constructed by changing the choice of $a$. 
\end{remark}

An immediate corollary of Theorem \ref{thm: intro lyap} is the following blow up result. 
\begin{corollary}\label{cor: blw up}
  Consider  $\omega_0\in H^{s}$ with $s>1$ verifying \eqref{eq:decay fourier hyp} then $\left\Vert D\Phi_t-Id\right\Vert_{\omega_0}$ blows up at least like $t^{\frac{1}{3}}$ in infinite time. 
\end{corollary}

We note that by Proposition \ref{prop: charac norm omeg} that if $dr_{\omega_0}(\epsilon)\underset{\epsilon \to 0}{\sim}\epsilon^s$ then $\left\Vert \cdot \right\Vert_{B^s_{2,\infty}}$ and $\left\Vert \cdot \right\Vert_{\omega_0}$ are equivalent. In particular the previous corollary gives the growth of $\left\Vert \Phi_t-Id \right\Vert_{B^{s+1}_{2,\infty}} $ for $\omega_0\in B^s_{2,\infty}$ with $dr_{\omega_0}(\epsilon)\underset{\epsilon \to 0}{\sim}\epsilon^s$. 
\subsection{Active scalar equations}
    The construction given in Theorem \ref{thm: intro lyap} is not special to the 2d Euler equations but can be adapted to a large class of active scalar equations as for example the generalised SQG equations given by 
    \begin{equation}
    \label{eq: sqg1} \partial_t \Theta + u\cdot\nabla \Theta=0,
\end{equation}
\begin{equation}\label{eq: sqg2}
    u=\nabla^\perp (-\Delta)^{-\frac{\alpha}{2}} \Theta.
\end{equation} 
Note that for $\alpha=2$ we get the 2d Euler equations. The results in this paper generalise in verbatim to give the following.
\begin{theorem}\label{thm:sqg}
     Consider $\alpha>1$, $s>3-\alpha$, $\Theta_0\in H^{s}\left( \mathbb{R}^2\right)$ and $\Theta\in C\left([0,T],H^{s}\left( \mathbb{R}^2\right)\right)$ the unique solution of \eqref{eq: sqg1}-\eqref{eq: sqg2} with initial data $\Theta_0$ for some $T\geq \frac{C_\delta}{\left\Vert \Theta_0 \right\Vert_{H^{3-\alpha-\delta}}}$ with $0<\delta<s+\alpha-3$ and $C_\delta>0$ is a constant depending only on $\delta$. Suppose moreover that $\Theta_0$ verifies \eqref{eq:decay fourier hyp}, then there exists a constant $C_{\Theta_0}$ such that for $0\leq t\leq T$
    \[
    \left\Vert \Theta(t) \right\Vert_{\Theta_0}+\left\Vert D\Phi_t -Id\right\Vert_{\Theta_0}\leq  \exp\left(C_{\Theta_0} \int^t_0\left\Vert\nabla u(s)\right\Vert_{L^\infty}ds \right).
    \]
    Moreover fix $\chi(\xi)\in C^\infty_0( \mathbb{R}^2\setminus B(0,1))$ then there exists a constant $C_{\Theta_0}$ such that for $\epsilon \geq 0$
\begin{multline*}
 \frac{d}{dt}\left((-\Delta)^{\frac{\alpha-2}{2}}\nabla \times T_{[D\Phi_t]^{-1}}\Phi_t,\chi(\epsilon D)\Theta_0\right)_{L^2}=-\left\Vert T_{\frac{\left\vert\xi\right\vert^{\frac{\alpha}{2}}}{\left\vert [D\Phi_t]^{-1}\xi\right\vert^{\frac{\alpha}{2}}}}\chi(\epsilon D)\Theta_0\right\Vert^2_{L^2}\\
 + O\left(C_\delta \exp\left(C_{\Theta_0} \int^t_0\left\Vert\nabla u(s)\right\Vert_{L^\infty}ds \right)\epsilon^{\min(s+\alpha-3-\delta,1,\alpha-1)}  dr_{\Theta_0}(\epsilon)^2\right),
 \end{multline*}
for all $0<\delta<s+\alpha-3$ and $C_\delta>0$ is constant depending only on $\delta$.
\end{theorem}

\subsection{Organisation of the paper}
In Section \ref{sec:prop exact smooth} we give the proof of Theorem \ref{thm: specific smooth prop}. In Section \ref{sec: Wig tran} we discuss the rates of convergence of semi-classical measures and show that they are exactly given by $dr_f$. In Section \ref{sec: lyp} we give the proof of Theorem \ref{thm: intro lyap} and Corollary \ref{cor: blw up}. Finally in Appendix \ref{sec: micr anal} we give a review of the microlocal analysis notions needed in this article.
\section*{Acknowledgements}
The author acknowledges funding from the UKRI grant SWAT. I would like to thank T. Alazard for suggesting to go beyond the Besov setting of the previous construction and his generosity in sharing techniques and knowledge, J. Bedrossian for having suggested to look for singularity formation using Shnirelman's construction, T. M. Elgindi for helpful discussions and remarks that helped the work take shape and N. Tzvetkov for an insightful discussion that helped better my understanding of the new norm used here. The starting idea of this work emanated from T. Drivas's suggestion and encouragement to study A. Shnirelman's construction, he also contributed valuable input along the realisation of this project for which I am deeply indebted.
\section{Exact smoothness propagation in 2d Euler}\label{sec:prop exact smooth}
In this section we give the proof of Theorem \ref{thm: specific smooth prop}. First for $\left\Vert \omega\right\Vert_{\omega_0}$, we paralinearise the Euler equation to get 
\[
\partial_t \omega+T_{u}\cdot \nabla \omega=-T_{\nabla \omega}\cdot u-R(u,\nabla \omega),
\]
thus commuting with $\chi(\epsilon D)$ we get
\[
\partial_t \chi(\epsilon D)\omega+T_{u}\cdot \nabla \chi(\epsilon D) \omega=-\left[\chi(\epsilon D),T_{u}\cdot \nabla\right]\omega-\chi(\epsilon D)T_{\nabla \omega}\cdot u-\chi(\epsilon D)R(u,\nabla \omega).
\]
Now by the continuity of paradifferential operators given in Theorem \ref{thm: para continuity}, combined with the symbolic calculus given in Theorem \ref{thm: symb calc para}, the spectral localisation property of paradifferential operators \ref{prop: para act spctrm} we see that 
\[
\left\Vert \left[\chi(\epsilon D),T_{u}\cdot \nabla\right]\omega \right\Vert_{L^2}\leq C \left \Vert Du\right\Vert_{L^\infty}dr_{\omega_0}(\lambda \epsilon)\leq C \left \Vert Du\right\Vert_{L^\infty}\left\Vert\omega\right\Vert_{\omega_0}dr_{\omega_0}(\lambda \epsilon),
\]
thus 
\[
\left\Vert \left[\chi(\epsilon D),T_{u}\cdot \nabla\right]\omega \right\Vert_{L^2}\leq C_{\omega_0} \left \Vert Du\right\Vert_{L^\infty}\left\Vert\omega\right\Vert_{\omega_0}dr_{\omega_0}(\epsilon).
\]
The previous estimate holds analogously for $\chi(\epsilon D)T_{\nabla \omega}\cdot u$ and $\chi(\epsilon D)R(u,\nabla \omega)$ and the desired result follows from a standard energy estimate.  To get the estimate on $\left\Vert D\Phi_t-x\right\Vert_{\omega_0}$ we need a more refined characterisation of $\left\Vert \cdot \right\Vert_{\omega_0}$ and the action of multiplication and composition on this norm.
\subsection{Analysis of the norm $\left\Vert\cdot  \right\Vert_{\omega_0}$}
First we show that \eqref{eq:decay fourier hyp} is equivalent to having finite smoothness.
\begin{proposition}\label{prop: eqv fr dec}
    Consider $f\in L^2(\mathbb{R}^2)$ verifying \eqref{eq:decay fourier hyp} then there exits $s> 0$ such that $f\notin H^s(\mathbb{R}^2)$. Conversely if $f\in L^2(\mathbb{R}^2)$ and there exists $s>0$ such that $f\notin H^s(\mathbb{R}^2)$ then $f$ verifies \eqref{eq:decay fourier hyp} moreover $C_f$ can be chosen in the form 
    \[
    C_f(\lambda)\leq C_\delta \left\Vert f \right\Vert_{H^{s_f}}  \lambda^{s_f+\delta},
    \]
    where $s_f\geq 0$ is the largest index such that $f\in H^{s_f}(\mathbb{R}^2)\setminus H^{s_f+\delta}(\mathbb{R}^2) $ for all $\delta>0$.
\end{proposition}
\begin{proof}[Proof of Proposition \ref{prop: eqv fr dec}]
The proof follows from the following lemma.
    \begin{lemma}\label{lem:algb dec}
    Consider a continuous decreasing function $F:\mathbb{R}_+\to \mathbb{R}_+$ such that $\displaystyle \lim_{x\to +\infty}F(x)=0$. Suppose that for all $\lambda\geq 1$, $\displaystyle C(\lambda)=\inf_{x\in \mathbb{R}_+}\frac{F(\lambda x)}{F(x)}>0$. Then there exists $\alpha$ and $ C_\alpha\geq 0$ such that $\displaystyle  C(\lambda) \geq \frac{C_\alpha}{1+\lambda^{\alpha}}$. 
Conversely if there exits $\alpha>0$ such that $F(x)\geq \frac{C_\alpha}{1+x^{\alpha}}$. Then for all $\lambda\geq 1$, $\displaystyle C(\lambda)=\inf_{x\in \mathbb{R}_+}\frac{F(\lambda x)}{F(x)}>0$.
\end{lemma}
\begin{proof}[Poof of Lemma \ref{lem:algb dec}]
    We note that $C$ is a decreasing function that goes from $1$ to $0$ which moreover verifies a Cauchy functional inequality of the form 
    \[
    C(\lambda_1)C(\lambda_2)\leq C(\lambda_1\lambda_2).
    \]
    For $x\geq 0$ we define $f(x)=-\ln(C(e^{x}))$, which is well defined as $C(\lambda)>0$. Then we get that $f$ is a positive increasing subadditive function
    \[
    f(x+y)\leq f(x)+f(y),
    \]
    thus $\lim_{x\to +\infty }\frac{f(x)}{x}=\inf_{x\geq 0}\frac{f(x)}{x}=\alpha\geq 0$ which gives the desired result.

For the second part of the lemma we proceed by contradiction and suppose that there exists $\lambda \geq 1$ and a sequence $x_n\underset{n\to +\infty}{\to}+\infty$ such that 
    \[
    \frac{F(\lambda x_n)}{F(x_n)}\underset{n\to+\infty}{\to}0.
    \]
     By hypothesis there exists $\epsilon_n$ such that $F(x_n)\leq x_n^{-\alpha+\epsilon_n}$ thus \[\frac{F(\lambda x_n)}{F(x_n)}\geq \lambda^\alpha \underbrace{x_n^{\epsilon_n}}_{\geq 1}\geq\lambda^\alpha,\]
    which is a contradiction. 
\end{proof}
\end{proof}
Next we give a characterisation of $\left\Vert \cdot \right\Vert_{\omega_0}$.
\begin{proposition}\label{prop: charac norm omeg}
    Consider $\omega_0\in L^2(\mathbb{R}^2)$ verifying \eqref{eq:decay fourier hyp} then there exist a constant $C_{\omega_0}\geq 1$ such that 
\[
\frac{1}{C_{\omega_0}}\left\Vert f \right\Vert_{\omega_0}\leq \sup_{k\geq 0} \frac{1}{dr_{\omega_0(2^{-k})}}\left\Vert \Delta_k f \right\Vert_{L^2}:=\left\vert f \right\vert_{\omega_0}\leq C_{\omega_0}\left\Vert f \right\Vert_{\omega_0}.
\]
    where $\Delta_k f$ are the Littlewood-Paley projectors given in Definition \ref{def: LP decomp}. 
\end{proposition}
\begin{proof}
We introduce the norm 
\[
\sup_{k\geq 0} \frac{1}{dr_{\omega_0(2^{-k})}}\sum^{+\infty}_{i=k}\left\Vert \Delta_i f \right\Vert_{L^2}=\left\Vert f \right\Vert'_{\omega_0}
\]
then by the slow varying hypothesis on $dr_{\omega_0}$ given by \eqref{eq:decay fourier hyp}, $\left\Vert f \right\Vert'_{\omega_0}$ and $\left\Vert f \right\Vert_{\omega_0}$ are equivalent (we refer to Chapter 1 of \cite{taylor2000tools} for a more general treatment of spaces defined with a slow varying moduli). Next we have the immediate inequality $\left\vert f \right\vert_{\omega_0} \leq \left\Vert f \right\Vert'_{\omega_0}$ and conversely 
\[
\frac{1}{dr_{\omega_0(2^{-k})}}\sum^{+\infty}_{i=k}\left\Vert \Delta_i f \right\Vert_{L^2}\leq \left\vert f \right\vert_{\omega_0} \frac{1}{dr_{\omega_0(2^{-k})}}\sum^{+\infty}_{i=k}dr_{\omega_0}(2^{-i}).
\]
Now by Proposition \ref{prop: eqv fr dec}
\[
\sum^{+\infty}_{i=k}dr_{\omega_0}(2^{-i})=\sum^{+\infty}_{i=0}dr_{\omega_0}(2^{-i-k})\leq C_{\omega_0}dr_{\omega_0}(2^{-k})\underbrace{\sum^{+\infty}_{i=0}2^{-\alpha_{\omega_0} i}}_{<+\infty},
\]
which gives the desired result.
\end{proof}

\subsection{Actions of operators on $\left\Vert \cdot \right\Vert_{\omega_0}$}
We first give a lemma on the action of multiplication on $\left\Vert \cdot \right\Vert_{\omega_0}$.
\begin{lemma}\label{lem: comp mult}
    Consider $ f,g\in H^s\left(\mathbb{R}^2\right)$ for $s>1$ and $\omega_0\in H^\epsilon\left(\mathbb{R}^2\right)$ with $\epsilon>0$ and $\omega_0$ verifying hypothesis \eqref{eq:decay fourier hyp} and then there exists $C_{\omega_0}$ such that
    \[
    \left\Vert fg\right\Vert_{\omega_0}\leq \frac{C_{\omega_0}}{\epsilon} \left\Vert f \right\Vert_{L^\infty} \left\Vert g\right\Vert_{\omega_0}+\frac{C_{\omega_0}}{\epsilon} \left\Vert g \right\Vert_{L^\infty} \left\Vert f\right\Vert_{\omega_0}.
    \]
   More generally we have for the paraproduct and residual paraproduct given in Theorem \ref{thm: para product}:
   \[
    \left\Vert T_fg\right\Vert_{\omega_0}\leq C_{\omega_0} \left\Vert f \right\Vert_{L^\infty} \left\Vert g\right\Vert_{\omega_0} \text{ and }   \left\Vert |D|^{r}R(f,g)\right\Vert_{\omega_0}\leq \frac{C_{\omega_0,r}}{\epsilon} \left\Vert f \right\Vert_{W^{r,\infty}} \left\Vert g\right\Vert_{\omega_0},
   \]
   and for a paradifferential operator with symbol $a \in \Gamma^0_0(\mathbb{R}^2)$ given in Definition \ref{def: para symbol}
  \[ \left\Vert T_af\right\Vert_{\omega_0}\leq C_{\omega_0} M_0^0(a;2) \left\Vert g\right\Vert_{\omega_0}.
  \]
\end{lemma}
\begin{proof}
    We follow the classical approach in \cite{bahouri2011fourier,taylor2000tools}. For the product estimate we decompose $fg$ as
    \[
    fg=T_gf+T_fg+R(f,g),
    \]
    thus it suffices to prove the estimates on $T_f g$ and $R(f,g)$. For $T_f g$ we compute 
    \[
    \Delta_k(T_f g)=\sum^{k+N}_{\substack{i\geq 0 \\ i=k-N}}  \Delta_k\left(P_{\leq i-1}(D)f\Delta_i g\right),
    \]
    for a fixed $N$, thus by the slow varying property of $dr_{\omega_0}$ we get the desired bound on $|T_f g|_{\omega_0}$. Next for $R(f,g)$ we write for \[
    R(f,g)=\sum_{q}R_q \text{ with } R_{q}=\sum^{1}_{i=-1}\Delta_{q-i} f\Delta_{q}g
    \]
    and we note that $R_q$ is supported in a ball $2^{q}B(0,\lambda)$ for a fixed $\lambda$. We now need the following standard lemma which follows in verbatim from the proof of Lemmas 2.49 and 2.84 of \cite{bahouri2011fourier} and the slow varying property \eqref{eq:decay fourier hyp} of $dr_{\omega_0}$.
    \begin{lemma}\label{lem: space charac with balls}
        Consider $\omega_0\in H^\epsilon(\mathbb{R}^2)$ with $\epsilon>0$ verifying \eqref{eq:decay fourier hyp} then there exist a constant $C_{\omega_0}\geq 1$ such that 
\[
\frac{\epsilon}{C_{\omega_0}}\left\Vert f \right\Vert_{\omega_0}\leq \sup_{k\geq 0} \frac{1}{dr_{\omega_0(2^{-k})}}\left\Vert P_{\leq k}(D) f \right\Vert_{L^2}\leq \frac{C_{\omega_0}}{\epsilon}\left\Vert f \right\Vert_{\omega_0}.
\]
    where $P_{k}(D), \ P_{\leq k}(D)$ are Littlewood-Paley projectors are the Littlewood-Paley projectors given in Definition \ref{def: LP decomp}. 
    \end{lemma}    
   Applying the previous Lemma to $R_q$,
    \[
    \left\Vert  |D|^{r}R(f,g)\right\Vert_{L^2}\leq 2^{qr} \left\Vert \Delta_q g \right \Vert_{L^2}\sum^{1}_{i=-1} \left\Vert \Delta_{q-i}f \right\Vert_{L^\infty},
    \]
    giving again the stated bound on $|R(f,g)|_{\omega_0}$. For a $0$th order paradifferential operators following \cite{coifmandela,auscher1995paradifferential}, $T_a$ can be written as a sum of a rapidly decreasing sequence of ``elementary symbols", where an elementary symbol is of the form 
    \[
q(x,\xi)=\sum^{+\infty}_{k=0}T_{Q_k(x)}P_k(\xi),
    \]
    where $P_k$ are the Littlewood-Paley projector given in Definition \ref{def: LP decomp} and $Q_k(x)$ are uniformly bounded in $L^\infty$ which again gives the stated bound from the previous computations.
\end{proof}
Next we need a lemma on the action of composition on $\left\Vert \cdot \right\Vert_{\omega_0}$. We will give it through the operator of paracomposition given in Theorem \ref{thm: def para comp} as those estimates will be needed for the proof of Theorem \ref{thm: intro lyap}.
\begin{lemma}\label{lem:paracomp omeg norm}
 Consider $\omega_0\in L^2\left(\mathbb{R}^2\right)$ with $\omega_0$ verifying hypothesis \eqref{eq:decay fourier hyp}. Let $\chi:\mathbb{R}^d \rightarrow \mathbb{R}^d$ be a $C^1\left(\mathbb{R}^2\right)$ diffeomorphism with $D\chi \in W^{r,\infty}$, $r>0, r\notin \mathbb{N}$, take $u \in H^s(\mathbb{R}^2)$ and define   
\[\chi^* u=\sum_{k\geq 0}  \sum_{\substack{l\geq 0 \\ k-N \leq l \leq k+N}}P_l(D)(\Delta_k \omega_0)\circ \chi,
\] 
where $N \in \mathbb{N}^*$ is such that $2^{N}>\sup_{k,\mathbb{R}^d} \left\vert \Phi_k D\chi\right\vert^{-1}$ and $2^{N}>\sup_{k,\mathbb{R}^2} \left\vert \Phi_k D\chi\right\vert$. Then there exists an increasing function $C_{\omega_0}$ (increasing at most polynomially fast) such that   \[ \left\Vert \chi^* u\right\Vert_{\omega_0}\leq C_{\omega_0}(\left\Vert D\chi\right\Vert_{L^\infty},\left\Vert D\chi^{-1}\right\Vert_{L^\infty})\left\Vert u\right\Vert_{\omega_0}.\] 
 Now the composition $\omega_0 \circ \chi(x)$ can be decomposed as
 \[\omega_0 \circ \chi(x)=\chi^* \omega_0(x)+ T_{\nabla \omega_0\circ \chi}\cdot\chi(x)+ R(x),\] 
and the remainder verifies the estimate for an increasing function $C_{\omega_0}$ (increasing at most polynomially fast)  
\[  \left\Vert |D|^r R\right\Vert_{\omega_0} \leq C_{\omega_0,r}(\left\Vert D\chi\right\Vert_{W^{r,\infty}},\left\Vert D\chi^{-1}\right\Vert_{L^\infty})\left\Vert \omega_0\right\Vert_{\omega_0}. \]
\end{lemma}
\begin{proof}
The estimate on $\chi^*u$ follows exactly from the one on $T_f g$ as 
\[
\Delta_j \left(\chi^* \omega_0\right)=\sum_{\substack{l\geq 0 \\ j-N \leq l \leq j+N}}P_l(D)(\Delta_k \omega_0)\circ \chi.
\]
Next for the estimates on $R$ we will follow the presentation in Appendix A and B of chapter 2 of \cite{taylor2000tools} and write 
\[
R=R_1+R_2+R_3,
\]
with 
\[
R_1=\sum_{k\geq 0}  \sum_{\substack{l\geq 0 \\ k-N \leq l \leq k+N}}P_l(D)\left((\Delta_k \omega_0)\circ \chi-(\Delta_k \omega_0)\circ \left(P_{\leq k}(D)\chi\right)\right),
\]
\[
R_2=\sum_{k\geq N}  P_{\leq k-N}(D)[(\Delta_k \omega_0)\circ \left(P_{\leq k}(D)\chi\right)],
\]
and 
\[
R_3=\sum_{k\geq N}  \left(Id-P_{\leq k+N}(D)\right)[(\Delta_k \omega_0)\circ \left(P_{\leq k}(D)\chi\right)].
\]
We note that $R_1$ can be treated similarly to $\chi^* \omega_0$. For $R_2$ and $R_3$ we recall estimates $(B.15)$ and $(B.16)$ page 133 of \cite{taylor2000tools}: there exists $N$ depending on $\left\Vert D\chi \right\Vert_{L^\infty}$ and K depending on $\left\Vert D\chi \right\Vert_{W^{r,\infty}}$ and $\left\Vert D\chi^{-1} \right\Vert_{L^\infty}$ such that for $k\geq K$, $j\geq k+N, \nu \geq r$
\[
\left\Vert P_j(D)\left((\Delta_k \omega_0)\circ \left(P_{\leq k}(D)\chi\right)\right) \right\Vert_{L^2}\leq C_\nu\left(\left\Vert D\chi \right\Vert_{W^{r,\infty}},\left\Vert D\chi^{-1} \right\Vert_{L^\infty}\right)2^{-j\nu} 2^{k(\nu-r)}\left\Vert (\Delta_k \omega_0) \right\Vert_{L^2},
\]
and 
\[
\left\Vert P_{\leq k-N}(D)\left((\Delta_k \omega_0)\circ \left(P_{\leq k}(D)\chi\right)\right) \right\Vert_{L^2}\leq C_\nu\left(\left\Vert D\chi \right\Vert_{W^{r,\infty}},\left\Vert D\chi^{-1} \right\Vert_{L^\infty}\right) 2^{-kr}\left\Vert (\Delta_k \omega_0) \right\Vert_{L^2}.
\]
which again give the claim.
\end{proof}

Now we turn to the estimate on $\left\Vert D\Phi_t-Id\right\Vert_{\omega_0}$. For this we
start from 
\[
\frac{d}{dt}\Phi_t=u\circ \Phi_t ,
\]
and use the paracomposition operator to write

\[
\frac{d}{dt}\Phi_t=\Phi_t^*u+T_{Du\circ \Phi_t}\Phi_t+R.
\]
Applying Lemmas \ref{lem: comp mult} and \ref{lem:paracomp omeg norm} we get 
\[
\frac{d}{dt}\left\Vert D\Phi_t-Id  \right\Vert_{\omega_0}\leq C_{\omega_0}\left\Vert Du\right\Vert_{L^\infty}\left\Vert D\Phi_t-Id \right\Vert_{\omega_0}+C_{\omega_0}\left(\left\Vert D\Phi_t\right\Vert_{L^\infty},\left\Vert D\Phi_t^{-1}\right\Vert_{L^\infty}\right)\left\Vert Du \right\Vert_{\omega_0},
\]
giving again the stated bound.
\section{Convergence of semi-classical measures}\label{sec: Wig tran}
While not stated exactly in this form, Shnirelman observation behind the ``microlocal scalar product" is on the effect of regularity on the convergence rate of semi-classical measures. First let us introduce those notions and the convergence result on $L^2\left(\mathbb{R}^2\right)$. 
\subsection{Convergence in $L^2(\mathbb{R}^2)$}
Consider a pseudodifferential operator $a\in S^0\left(\mathbb{R}^2\right)$ and $u,v\in L^2\left(\mathbb{R}^2\right)$, then by the continuity of pseudodifferential operators on $L^2$ (see Theorem \ref{thm: pseudo continuity}) we see that 
\[
\left(a(x,\epsilon D)u,v\right)_{L^2(\mathbb{R}^2)}\underset{\epsilon \to 0}{\to }\int_{\mathbb{R}^2}a(x,0)u(x)\bar{v}(x)dx=\lambda_{u,v}(a),
\]
and $\lambda_{u,v}$ can be seen as measure on the space $S^0(\mathbb{R}^2)$. This measure was first introduced simultaneously by two independent procedures in \cite{gerard1991mesures} and \cite{lions1993mesures}. The above construction method is the procedure introduced in \cite{gerard1991mesures}. Note that $\lambda_{u,u}$ is always a positive measure, thus if $a(x,\xi)\geq 0$ and $\lambda_{u,u}(a)>0$ then for $\epsilon$ sufficiently small 
\[
\left(a(x,\epsilon D)u,u\right)_{L^2}=\int_{\mathbb{R}^2\times\mathbb{R}^2}a(x,\epsilon\xi)e^{ix\cdot \xi}\mathscr{F}(u)(\xi)u(x)dxd\xi>0.
\]
Note that the only property on $a$ we used in the proof of the previous theorem is the continuity of $a(x,D)=\mbox{Op}(a)$ on $L^2(\mathbb{R}^2)$ which holds more generally, see Theorem \ref{thm: para continuity}, for paradifferential operators with symbol $a\in \Gamma_0^0(\mathbb{R}^2)$ with regularised symbol $\sigma_a$ (see Definitions \ref{def: para symbol} and \ref{def: para op}). The cut-off used in the definition of paradifferential operators ensures that $\sigma_a(x,\xi)=0$ for $\xi \in B(0,b)$ for some $b>0$ fixed thus we have proved the following.
\begin{theorem}[From \cite{gerard1991mesures} and \cite{lions1993mesures}]\label{thm:l2 conv of meas psd}
    For $u,v\in L^2\left(\mathbb{R}^2\right)$ and $a\in S^0\left(\mathbb{R}^2\right)$ and $a'\in  \Gamma^0_0(\mathbb{R}^2) $ then 
    \[
    \left(a(x,\epsilon D)u,v\right)_{L^2(\mathbb{R}^2)}, \left(\sigma_{a'}(x,\epsilon D)u,v\right)_{L^2(\mathbb{R}^2)},
    \]
    converge respectively  to $\lambda_{u,v}(a)$ and $0$ in the limit $\epsilon \to 0$.
\end{theorem}

\subsection{The rate of convergence}
A first observation is that the rates of convergence of sequences given in Theorems \ref{thm:l2 conv of meas psd} depend on the regularity of $u,v$ more precisely by Remark $III.9$ of \cite{lions1993mesures} if $u\in H^s(\mathbb{R}^2),s>0$ then for $a\in S^0\left(\mathbb{R}^2\right)$, $\left(a(x,\epsilon D)u,u\right)_{L^2(\mathbb{R}^2\times \mathbb{R}^2)}-\lambda_{u,u}(a)=O(\epsilon^{2s})$. Shnirelman then showed in 
\cite{ShnLyap} that the same rate holds for $u\in B^s_{2,\infty}$ and he insight-fully observed that if $u\notin B^{s+\delta}_{2,\infty}(\mathbb{R}^2)$ for $\delta>0$ then the previous rate of convergence is essentially exact. While he did not state the second observation explicitly he uses it to prove Theorem 7.1 of \cite{ShnLyap}. We will extend the decay rate observations of \cite{lions1993mesures} and \cite{ShnLyap} to $L^2$ functions in general. We recall the definition for $u\in L^2\left(\mathbb{R}^2\right)$ 
\[
dr_u(\epsilon)=\left(\int_{\mathbb{R}^2\setminus B\left(0,\frac{1}{\epsilon}\right)}\left|\mathscr{F}(u)(\xi)\right|^2 d\xi\right)^{\frac{1}{2}},
\]
which is an increasing function of $\epsilon$ with $\displaystyle \lim_{\epsilon \to 0}  dr_u(\epsilon)=0$ and $\displaystyle \lim_{\epsilon \to +\infty}  dr_u(\epsilon)=\left\Vert u\right\Vert_{L^2(\mathbb{R}^2)}$.
\begin{theorem}\label{thm: rt of conv W meas para}
    Consider $u,v\in L^2(\mathbb{R}^2)$, then there exists universal constants $C_1,C_2, C_3>0$ such that for $a\in \Gamma_0^0(\mathbb{R}^2)$ and $\epsilon>0$
    \[
    \left\vert \left(\sigma_a(x,\epsilon D)u,v\right)_{L^2(\mathbb{R}^2)}\right\vert\leq C_1 M^0_0(a;2) dr_u(C_2\epsilon)dr_v(C_3\epsilon),
    \]
    where the semi-norms $M^{\cdot}_{\cdot}(\cdot;\cdot)$ are defined as in Definition \ref{def: para symbol}. 
    If there exists a $0$-homogeneous symbol $a_0\in \Gamma^0_0(\mathbb{R}^2)$ such that $a-a_0\in \Gamma^{-\alpha}_0(\mathbb{R}^2)$ for some $\alpha> 0$ then $a_0$ is called the principal symbol of $a$ and moreover we have
\[
    \left\vert \left(\sigma_{a-a_0}(x,\epsilon D)u,v\right)_{L^2(\mathbb{R}^2)}\right\vert \leq C_1 \epsilon^\alpha M^{-\alpha}_0(a-a_0;2) dr_u(C_2\epsilon)dr_v(C_3\epsilon).
\]
\end{theorem}
\begin{remark}
    The rates given in the previous theorem are optimal in the sense that by definition there exists a $C$ such that 
    \[
    \left\vert \left(\sigma_1(x,\epsilon D)u,u\right)_{L^2(\mathbb{R}^2)}\right\vert \gtrsim C_1 dr_{u}(C_2\epsilon)dr_{u}(C_3\epsilon) .
    \]
    Thus the previous theorem shows that the exact rate of convergence of semi-classical is an exact bilinear test of smoothness.
\end{remark}
\begin{proof}
    First noting that by the polarization formula for a sesquilinear form $s$
    \[
    s(u,v)=\frac{1}{4}\left(s(u+v,u+v)-s(u-v,u-v)+is(u+iv,u+iv)-is(u-iv,u-iv)\right),
    \]
    it suffices to work with $u=v$. First we study $\left(\sigma_a(x,\epsilon D)u,u\right)_{L^2(\mathbb{R}^2)}$ and write by Parseval's identity
    \begin{align*}
        \left\vert \left(\sigma_a(x,\epsilon D)u,u\right)_{L^2(\mathbb{R}^2)} \right\vert &=\left\vert \int_{\mathbb{R}^2}\mathscr{F}\left(\sigma_a(x,\epsilon D)u\right)(\xi)\mathscr{F}(u)(\xi)d\xi \right\vert
  \intertext{using the Cauchy-Schwarz inequality combined with the spectral localisation property of paradifferential operators given in Proposition \ref{prop: para act spctrm} we get that there exists a universal constant $C_1, C_2>0$ such that}   
   \left\vert \left(\sigma_a(x,\epsilon D)u,u\right)_{L^2(\mathbb{R}^2)} \right\vert &\leq C_1 \left\Vert \mathscr{F}\left(\sigma_a(x,\epsilon D)u\right)\right\Vert_{L^2\left(\mathbb{R}^2\setminus B\left(0,\frac{C_2}{\epsilon}\right)\right)}\left\Vert \mathscr{F}(u)\right\Vert_{L^2\left(\mathbb{R}^2\setminus B\left(0,\frac{C_2}{\epsilon}\right)\right)}
   \intertext{by the boundness of paradifferential operators on $L^2$, see Theorem \ref{thm: para continuity} we get, and again the spectral localisation property of paradifferential operators}
    \left\vert \left(\sigma_a(x,\epsilon D)u,u\right)_{L^2(\mathbb{R}^2)} \right\vert &\leq C_1 M^0_0(a;2)  dr_u(C_2\epsilon) dr_u(C_3\epsilon).
    \end{align*}
    To get the estimate on $ \left\vert \left(\sigma_{a-a_0}(x,\epsilon D)u,u\right)_{L^2(\mathbb{R}^2)}\right\vert$ we start from 
    \[
    \left\vert \left(\sigma_{a-a_0}(x,\epsilon D)u,u\right)_{L^2(\mathbb{R}^2)} \right\vert \leq  \left\Vert \mathscr{F}\left(\sigma_{a-a_0}(x,\epsilon D)u\right)\right\Vert_{L^2\left(\mathbb{R}^2\setminus B\left(0,\frac{C}{\epsilon}\right)\right)}\left\Vert \mathscr{F}(u)\right\Vert_{L^2\left(\mathbb{R}^2\setminus B\left(0,\frac{C}{\epsilon}\right)\right)},
    \]
  by the boundness of paradifferential operators in $\Gamma^{-\alpha}_0(\mathbb{R}^2)$ from $L^2(\mathbb{R}^2)$ to $H^\alpha(\mathbb{R}^2)$, see Theorem \ref{thm: para continuity} and again the spectral localisation property of paradifferential operators we get
  \[ \left\vert \left(\sigma_{a-a_0}(x,\epsilon D)u,u\right)_{L^2(\mathbb{R}^2)} \right\vert \leq C_1 M^{-\alpha}_0(a-a_0;2) \epsilon^\alpha dr_u(C_2\epsilon) dr_u(C_3\epsilon).\]
\end{proof}
The previous proof adapts in verbatim to give the following pseudodifferential version of the previous theorem. 
\begin{theorem}\label{thm: rt of conv W meas pseudo} Consider $u,v\in L^2(\mathbb{R}^2)$, $0<\delta<\epsilon$ and a function $\chi\in C^\infty_0\left(\mathbb{R}^2\setminus B(0,1)\right)$. Then there exists a universal constant $C$ and constant $C_\chi$, depending only on $\chi$, such that for $a\in S^0\left(\mathbb{R}^2\right)$ and $\epsilon>0$
    \[
    \left\vert \left(\chi( \epsilon D)a(x,  D)\chi( \epsilon D)u,v\right)_{L^2(\mathbb{R}^2)}\right\vert\leq C M^0_0(a;2) dr_u(C_\chi\epsilon)dr_v(C_\chi\epsilon),
    \]
    where the semi-norms $M^{\cdot}_{\cdot}(\cdot;\cdot)$ are defined as in Definition \ref{def: pseudo symb}. If there exists a zero homogeneous symbol $a_0$ such that $a_0\chi\in S^0\left(\mathbb{R}^2\right)$ and $(a-a_0)\chi\in S^{-\alpha}(\mathbb{R}^2)$ for some $\alpha> 0$ then $a_0$ is called the principal symbol of $a$ and moreover we have
\[
    \left\vert \left(\chi(\epsilon D)(a-a_0)(x,  D)\chi(\epsilon D)u,v\right)_{L^2(\mathbb{R}^2)}\right\vert \leq C \epsilon^\alpha M^{-\alpha}_0(a-a_0;2) dr_u(C_\chi\epsilon)dr_v(C_\chi\epsilon).
\]
\end{theorem}
\begin{remark}
 Theorem $3.1$ in \cite{ShnLyap} corresponds to the case in the previous theorem where $dr_\epsilon(u),dr_\epsilon(v)$ are of order $\epsilon^s$ for some $s\in \mathbb{R}$.
\end{remark}

\section{The monotone quantity in 2d Euler}\label{sec: lyp}
With Theorems \ref{thm: rt of conv W meas pseudo} and \ref{thm: rt of conv W meas para} as well as Theorem \ref{thm: specific smooth prop} we transform the qualitative computation in the proof of Theorem $5.1$ of \cite{ShnLyap} into the quantitative results given in Theorems \ref{thm: intro lyap}.
\subsection{Proof of Theorem \ref{thm: intro lyap}}
We will prove a more general form of Theorem \ref{thm: intro lyap} where $\chi(\epsilon D)$ is replaced by a paradifferential operator $\sigma_a(x,\epsilon D)$ with $a\in \Gamma_1^0$.  We start by computing the time derivative of $T_{ D\Phi_t}\left(\Phi_t-Id\right)$ where the paraproduct here should be understood component by component. We compute
\[
\partial_t\left(T_{\left[D\Phi_t\right]^{-1}} \Phi_t\right)=T_{\partial_t\left[D\Phi_t\right]^{-1}}\Phi_t+T_{\left[D\Phi_t\right]^{-1}} u\circ \Phi_t.
\]
For a matrix valued function $A_t$ we have $\partial_t \left(A_t^{-1}\right) =-A_t^{-1}    \partial_t A_t   A_t^{-1}$, thus
\begin{align*}
\partial_t\left[D\Phi_t\right]^{-1}&=-\left[D\Phi_t\right]^{-1} \left(\partial_t D\Phi_t\right) \left[D\Phi_t\right]^{-1}\\
&=-\left[D\Phi_t\right]^{-1} \left( D\left(u\circ \Phi_t\right)\right) \left[D\Phi_t\right]^{-1}\\
&=-\left[D\Phi_t\right]^{-1} Du\circ \Phi_t.
\end{align*}
Which gives 
\[
\partial_t\left(T_{\left[D\Phi_t\right]^{-1}} \Phi_t\right)=-T_{\left[D\Phi_t\right]^{-1} Du\circ \Phi_t}\Phi_t+T_{\left[D\Phi_t\right]^{-1}} u\circ \Phi_t.
\]
Using the paracomposition operator defined in \ref{thm: def para comp} we define $R_1$ by 
\begin{equation}\label{eq: def R_1}
u\circ \Phi_t=\Phi_t^*u+T_{Du\circ \Phi_t} \Phi_t+R_1=\Phi_t^*u+T_{Du\circ \Phi_t} \Phi_t+R_1.
\end{equation}
Note that the paracomposition operator should also be understood component by component. Thus we get the key cancellation in the second term on the left hand side
\[
\partial_t\left(T_{\left[D\Phi_t\right]^{-1}} \Phi_t\right)=T_{\left[D\Phi_t\right]^{-1}}\Phi_t^*u+\left(T_{\left[D\Phi_t\right]^{-1}}T_{Du\circ \Phi_t}-T_{\left[D\Phi_t\right]^{-1} Du\circ \Phi_t}\right)\Phi_t+T_{\left[D\Phi_t\right]^{-1}} R_1.
\]
Now we write 
\[
u=\nabla^{\perp}\Delta^{-1}\left(\omega_0\circ \Phi_t^{-1}\right),
\]
and again using the paracomposition operator we define $R_2$ by
\begin{equation}\label{eq: def R_2}
\omega_0\circ \Phi_t^{-1}={\Phi_t^{-1}}^*\omega_0+T_{\left(\nabla \omega_0\right)\circ \Phi^{-1}_t} \cdot\Phi^{-1}_t+R_2.
\end{equation}
Next we define $R_3$ by
\begin{equation}\label{eq: def R_3}
\nabla \times T_{\left[D\Phi_t\right]^{-1}} \Phi_t^*\left(\nabla^{\perp}\Delta^{-1}\left({\Phi_t^{-1}}^*\omega_0\right)\right)=T_{\frac{\left\vert\xi\right\vert^2}{\left\vert [D\Phi_t]^{-1}\xi\right\vert^2}}\omega_0+R_3.
\end{equation}
Putting the previous computations together we finally get the paralinearised evolution equation on the flow
\begin{equation}\label{eq: renorm evo flow}
\partial_t\left(\nabla \times\left(T_{\left[D\Phi_t\right]^{-1}} \Phi_t\right)\right)=T_{\frac{\left\vert\xi\right\vert^2}{\left\vert [D\Phi_t]^{-1}\xi\right\vert^2}}\omega_0+\boldsymbol{R}_1+\boldsymbol{R}_2+R_3,
\end{equation}
with 
\[
\boldsymbol{R}_1=\nabla \times T_{\left[D\Phi_t\right]^{-1}}R_1+\nabla \times T_{\left[D\Phi_t\right]^{-1}} \nabla^{\perp}\Delta^{-1}R_2,
\]
and 
\[
\boldsymbol{R}_2=\nabla \times T_{\left[D\Phi_t\right]^{-1}}\nabla^{\perp}\Delta^{-1}T_{\left(\nabla \omega_0\right)\circ \Phi^{-1}_t} \Phi^{-1}_t +\nabla\times \left(T_{\left[D\Phi_t\right]^{-1}}T_{Du\circ \Phi_t}-T_{\left[D\Phi_t\right]^{-1} Du\circ \Phi_t}\right)\Phi_t.
\]
We now proceed to estimate each term in $\left(\sigma_a(x,\epsilon D)(\boldsymbol{R}_1+\boldsymbol{R}_2+R_3),\sigma_a(x,\epsilon D)\omega_0\right)_{L^2(\mathbb{R}^2)}$, note that it suffices to estimate $dr_{\boldsymbol{R}_1}(\epsilon), \ dr_{\boldsymbol{R}_2}(\epsilon)$ and $dr_{R_3}(\epsilon)$. Henceforth we fix $s-1>\delta>0$. For $R_3$ we note that by the pull-back formula of paradifferential operators by the paracomposition operator given in \eqref{eq:paracom exp} as well as the composition property of paradifferential operators in Theorem \ref{thm: symb calc para} that by construction $T_{\frac{\left\vert\xi\right\vert^2}{\left\vert [D\Phi_t]^{-1}\xi\right\vert^2}}$ is the principal symbol of 
\[
\nabla \times T_{\left[D\Phi_t\right]^{-1}} \Phi_t^*\left(\nabla^{\perp}\Delta^{-1}\left({\Phi_t^{-1}}^*\omega_0\right)\right).
\]
 Thus by the spectral localisation property of paradifferential operators in Proposition \ref{prop: para act spctrm}, the continuity of paraproducts in Theorem \ref{thm: para product} we have for a universal increasing function $C_\delta$ growing at most polynomially fast
\[
dr_{R_3}(\epsilon)\leq  \epsilon^{\min(s-1-\delta,1)} C_\delta\left(\left\Vert D\Phi_t\right\Vert_{L^\infty},\left\Vert [D\Phi_t]^{-1}\right\Vert_{L^\infty}\right)  \left\Vert D\Phi_t\right\Vert_{W^{\min(s-1-\delta,1),\infty}}dr_{\omega_0}(\lambda \epsilon).
\]
Next $\boldsymbol{R}_2$ is treated in the same fashion as $R_3$. First $T_{\left(\nabla\omega_0\right)\circ \Phi_t^{-1}}$ is treated as an operator of order $\max(2-s,0)$ to give 
\[
dr_{\nabla \times T_{\left[D\Phi_t\right]^{-1}}\nabla^{\perp}\Delta^{-1}T_{\left(\nabla \omega_0\right)\circ \Phi^{-1}_t}\cdot \Phi^{-1}_t}(\epsilon)\leq C_\delta  \epsilon^{\min(s-1-\delta,1)} \left\Vert [D\Phi_t]^{-1}\right\Vert_{L^\infty} \left\Vert \omega_0\right\Vert_{H^s}  dr_{D\Phi^{-1}_t}(\lambda \epsilon),
\]
and $T_{\left[D\Phi_t\right]^{-1}}T_{Du\circ \Phi_t}-T_{\left[D\Phi_t\right]^{-1} Du\circ \Phi_t}$ is of order $-\max(s-1-\delta,1)$ which gives
\begin{multline*}
dr_{\nabla\times \left(T_{\left[D\Phi_t\right]^{-1}}T_{Du\circ \Phi_t}-T_{\left[D\Phi_t\right]^{-1} Du\circ \Phi_t}\right)\Phi_t}(\epsilon)\\ \leq C_\delta  \epsilon^{\min(s-1-\delta,1)}  \left\Vert D\Phi_t\right\Vert_{W^{\min(s-1-\delta,1),\infty}} \left\Vert Du\right\Vert_{W^{\min(s-1-\delta,1),\infty}} dr_{D\Phi_t}(\lambda \epsilon).
\end{multline*}
Finally for $\boldsymbol{R}_1$ we have the terms that bring into play the residual term of the paracomposition operator thus we have the gain of $\min(s-1-\delta,1)$ derivative and thus the $\epsilon^{\max(s-1-\delta,1)}$ factor by Lemma \ref{lem:paracomp omeg norm} to get
\[
dr_{ \nabla \times T_{\left[D\Phi_t\right]^{-1}}R_1}(\epsilon) \leq C_\delta\left(\left\Vert D\Phi_t\right\Vert_{W^{\min(s-1-\delta,1),\infty}},\left\Vert [D\Phi_t]^{-1}\right\Vert_{L^\infty}\right) \epsilon^{\min(s-1-\delta,1)} dr_{Du}(\epsilon),
\]
and
\[
dr_{\nabla \times T_{\left[D\Phi_t\right]^{-1}} \nabla^{\perp}\Delta^{-1}R_2}(\epsilon)\leq C_\delta\left(\left\Vert D\Phi_t\right\Vert_{W^{\min(s-1-\delta,1),\infty}},\left\Vert [D\Phi_t]^{-1}\right\Vert_{L^\infty}\right) \epsilon^{\min(s-1-\delta,1)} dr_{\omega_0}(\epsilon).
\]
Plugging the estimates given in Theorem \ref{thm:welp 2d eul} and \ref{thm: specific smooth prop} we get the desired result on $\boldsymbol{R}_1,\boldsymbol{R}_2$ and $R_3$.
Finally to get the result in the form stated in Theorem \ref{thm: intro lyap} we see that
\[\left[T_a,T_{\frac{\left\vert\xi\right\vert^2}{\left\vert [D\Phi_t]^{-1}\xi\right\vert^2}}\right] , \ T_{\frac{\left\vert\xi\right\vert^2}{\left\vert [D\Phi_t]^{-1}\xi\right\vert^2}}-T_{\frac{\left\vert\xi\right\vert}{\left\vert [D\Phi_t]^{-1}\xi\right\vert}}T_{\frac{\left\vert\xi\right\vert}{\left\vert [D\Phi_t]^{-1}\xi\right\vert}} \text{ and } T_{\frac{\left\vert\xi\right\vert}{\left\vert [D\Phi_t]^{-1}\xi\right\vert}}-\left(T_{\frac{\left\vert\xi\right\vert}{\left\vert [D\Phi_t]^{-1}\xi\right\vert}}\right)^t \] are all smoothing operators of order $-\min(s-1-\delta,1)$ where $T_a^{t}$ is the adjoint of $T_a$.

\subsection{Proof of Corollary \ref{cor: blw up}}
We suppose without loss of generality that $\left\Vert D\Phi_t \right\Vert_{L^\infty}<+\infty$ thus there exists $K>0$ such that for all $t\geq 0$
\[
\frac{\left\vert \xi\right\vert}{\left\vert [D\Phi_t]^{-1}\xi\right\vert}\geq K.\]
Thus for all $t$ there exists an $a$ such that 
\[
\left\Vert T_{\frac{\left\vert\xi\right\vert}{\left\vert [D\Phi_t]^{-1}\xi\right\vert}}a(x,\epsilon)\omega_0\right\Vert^2_{L^2(\mathbb{R}^2)}\geq C dr_{\omega_0}(\epsilon)^2,
\]
 thus
\[
\sup_{\epsilon\geq 0} \sup_{\substack{a\in C_0^\infty\left(\mathbb{R}^2\right)\\ M^0_1(a)\leq 1}}  \frac{ \left(a(x,\epsilon D)\nabla \times T_{[D\Phi_t]^{-1}}\Phi_t,\chi(\epsilon D)\omega_0\right)}{dr_{\omega_0}(\epsilon)^2}\geq Ct.
\]
Thus $\left\Vert D\Phi_t-Id \right\Vert_{\omega_0}$ blows up at least linearly when $t\to +\infty$. To get the more general $t^{\frac{1}{3}}$ we note that with an analogous reasoning we have the differential inequality 
\[
\left\Vert D\Phi_t\right\Vert_{L^\infty}\left\Vert D\Phi_t-Id \right\Vert_{\omega_0} \geq C \int_0^t\frac{1}{\left\Vert D\Phi_s\right\Vert_{L^\infty}}ds.
\]

\appendix
\section{Notions of mircolocal analysis}\label{sec: micr anal}
At the heart of the construction in Theorem \ref{thm: intro lyap} is the use of microlocal analysis, in particular Shnirelman's construction can be seen as a consequence of the algebra property of pseudodifferential and paradifferential operators combined with the fact that restricted to principal symbols this algebra becomes commutative. In this section we will give a simple yet complete recount of all the concepts and classical results that are needed to carry out the construction. All of the results are taken from the excellent monographs \cite{alinhac2007pseudo,hormander1997lectures,metivier2008differential,taylor2000tools}
\subsection{The Littlewood-Paley decomposition}
\begin{definition}[Littlewood-Paley decomposition]\label{def: LP decomp}
Pick $P_0\in C^\infty_0(\mathbb{R}^2)$ so that: $$P_0(\xi)=1 \text{ for } \left\vert \xi \right\vert <1 \text{ and } P_0(\xi)=0 \text{ for } \left\vert \xi \right\vert >2. $$ We define a dyadic decomposition of unity by:
\[ \text{for } k \geq 1, \ P_{\leq k}(\xi)=P_0(2^{-k}\xi), \ P_k(\xi)=P_{\leq k}(\xi)-P_{\leq k-1}(\xi). \]
 Thus,\[ P_{\leq k}(\xi)=\sum^k_{j=0}P_j(\xi) \text{ and } 1=\sum_{j=0}^\infty P_j(\xi). \]
 Introduce the operator acting on $\mathscr{S} '(\mathbb{R}^2)$: 
 \[P_{\leq k}(D)u=\mathscr{F}^{-1}(P_{\leq k}(\xi)u) \text{ and } u_k=P_k(D)u=\mathscr{F}^{-1}(P_k(\xi)u).\]
 Thus,
 \[u=\sum^{+\infty}_{k=0} u_k.\]
 Finally put for $k\geq 1, C_k=\mbox{supp} \ P_k$ the set of rings associated to this decomposition.
\end{definition}

An interesting property of the Littlewood-Paley decomposition is that even if the decomposed function is merely a distribution the terms of the decomposition are regular, indeed they all have compact spectrum and thus are entire functions. On classical functions spaces this regularization effect can be ``measured" by the following inequalities due to Bernstein.

\begin{proposition}[Bernstein's inequalities]\label{prop: bernstein1}
Suppose that $a\in L^p(\mathbb{R}^2)$ has its spectrum contained in the ball $\{\left\vert \xi \right\vert \leq \lambda\}$. 

Then $a\in C^\infty$ and for all $\alpha \in  \mathbb{N}^d$ and $1\leq p \leq q \leq +\infty$, there is $C_{\alpha,p,q}$ (independent of $\lambda$) such that 
\[\left\Vert \partial_x^{\alpha} a \right\Vert _{L^q} \leq C_{\alpha,p,q} \lambda^{\left\vert \alpha\right \vert +\frac{d}{p}-\frac{d}{q}}\left\Vert a \right\Vert _{L^p}.\]
In particular,
\[\left\Vert \partial_x^{\alpha} a \right\Vert _{L^q} \leq C_{\alpha} \lambda^{\left\vert \alpha\right \vert }\left\Vert a \right\Vert _{L^p}, \text{ and for $p=2$, $p=\infty$}\]
\[\left\Vert a \right\Vert _{L^\infty}\leq C \lambda^{\frac{d}{2}} \left\Vert a \right\Vert _{L^2}.\]
If moreover a has its spectrum included in $ \{0<\mu \leq \left\vert \xi \right\vert \leq \lambda\}$ then:
\[
 C_{\alpha,q}^{-1} \mu^{\left\vert \alpha\right \vert }\left\Vert a \right\Vert _{L^q}\leq \left\Vert \partial_x^{\alpha} a \right\Vert _{L^q} \leq C_{\alpha,q} \lambda^{\left\vert \alpha\right \vert }\left\Vert a \right\Vert _{L^q}.
\]
\end{proposition}

 \subsection{Pseudodifferential calculus}
We introduce here the basic definitions and symbolic calculus results. We first introduce the classes of regular symbols.

\begin{definition}\label{def: pseudo symb}
Given $m \in \mathbb{R}$ we denote the symbol class $S^m(\mathbb{R}^2 )$ as the set of all $a\in C^\infty(\mathbb{R}^2 \times \mathbb{R}^2)$ such that for all $\alpha,\beta \in \mathbb{N}^2$ we have the estimate:
\[\left\vert\partial^{\alpha}_x\partial^{\beta}_\xi a(x,\xi)\right\vert\leq C_{\alpha,\beta}(1+\left\vert\xi\right\vert)^{m- \beta}.\]
$S^m(\mathbb{R}^2 )$ is a Fr\'echet space with the topology defined by the family of semi-norms:
\[M^m_{\alpha}(a;\beta)=\sup_{i\leq \left\vert\alpha\right\vert,j\leq \left\vert \beta\right\vert}\sup_{\mathbb{R}^2 \times \mathbb{R}^2}\left\vert\partial^{i}_x\partial^{j}_\xi a(x,\xi)(1+\left\vert\xi\right\vert)^{ j-m}\right\vert.\]
\end{definition}
Given a symbol $a\in S^{m}(\mathbb{R}^2 )$, we define the pseudodifferential operator:
\[\mbox{Op} (a)u(x)=a(x,D)u(x)=\frac{1}{(2\pi)^2}\int_{\mathbb{R}^2}e^{ix.\xi}a(x,\xi)\mathscr{F}(u)(\xi)d\xi. \]
\begin{theorem}\label{thm: pseudo continuity}
Consider $m \in \mathbb{R}$ then for $a\in S^m\left(\mathbb{R}^2\right)$, $\mbox{Op} (a)$ is of order $m$, more precisely for $s \in \mathbb{R}$ there exists a constant $C_s$ such that:
\[\left\Vert \mbox{Op} (a) \right\Vert_{H^{s} \rightarrow H^{s-m}}\leq C_s M^m_{\lceil s \rceil}(a,2).\]
\end{theorem}
We will now present the main results in symbolic calculus associated to pseudodifferential operators.
\begin{theorem} \label{thm: symb calc pseudp}
Consider two real numbers $m,m' \in \mathbb{R}$ and two symbols $a \in S^m\left(\mathbb{R}^2\right)$and $b \in S^{m'}(\mathbb{R}^2)$ then we have the following.
\begin{itemize}
\item Composition: $\mbox{Op} (a)\circ \mbox{Op} (b)$ is a pseudodifferential operator of order $m+m'$ with symbol $a \otimes b$ defined by:
\[a \otimes  b(x,\xi)=\frac{1}{(2\pi)^2}\int_{\mathbb{R}^2\times \mathbb{R}^2}e^{i(x-y)\cdot(\xi-\eta)} a(x,\eta)b(y,\xi)dyd\eta.\]
Moreover there exists a constant $C_s>0$ such that for $k\in \mathbb{N}$, $s \in \mathbb{R}$ and $(p,q)\in [1,+\infty]$,
\begin{multline*}\left\Vert\mbox{Op} (a)\circ \mbox{Op} (b)(x,\xi)-\mbox{Op} \left(\sum_{\left\vert \alpha\right\vert<k}\frac{1}{i^{\left\vert\alpha\right\vert}\alpha!}(\partial^\alpha_\xi a(x,\xi))(\partial^\alpha_x b(x,\xi))\right)\right\Vert_{H^{s} \rightarrow H^{s-m+k}} \\ \leq C_s (M^m_{k+\lceil s\rceil} (a;2) M^{m'}_{\lceil s\rceil}(b;k+2)+M^m_{\lceil s\rceil} (a;k+2) M^{m'}_{k+\lceil s\rceil}(b;2)).
\end{multline*}
\item Adjoint: The adjoint operator of $\mbox{Op} (a)$, that we will denote with $\mbox{Op} (a)^t$ to avoid confusion with the pullback operator defined in this work, is a pseudodifferential operator of order m with  symbol $a^t$ defined by:
\[a^t(x,\xi)=\frac{1}{(2\pi)^2}\int_{\mathbb{R}^2\times \mathbb{R}^2}e^{-iy\cdot\eta} \bar{a}(x-y,\xi-\eta)dyd\eta\]
Moreover there exists a constant $K>0$ such that for $k\in \mathbb{N}$, $s \in \mathbb{R}$ and $(p,q)\in [1,+\infty]$,
\[\left\Vert \mbox{Op} (a^t)(x,\xi)-\mbox{Op} \left(\sum_{\left\vert\alpha\right\vert <k}\frac{1}{i^{\left\vert\alpha\right\vert}\alpha!}(\partial^\alpha_\xi \partial^\alpha_x  \bar{a}(x,\xi))\right)\right\Vert_{H^{s} \rightarrow H^{s-m+k}} \leq C_s M^m_{k+\lceil s\rceil} (a;k+2).\]
\item Change of variables: consider a smooth diffeomorphism $\chi:\mathbb{R}^2\to \mathbb{R}^2$ such that $D\chi \in W^{k,\infty}\left(\mathbb{R}^2\right)$ for all $k\in \mathbb{N}$. Then defining
\[a^*(x,\xi)=e^{-ix\cdot\xi}\int_{\mathbb{R}^2\times \mathbb{R}^2} a(\chi(x),\eta) e^{i(\chi(x)-\chi(y))\cdot\eta+iy\cdot\xi}\left|  D\chi(y)\right|dyd\eta \in  S^m(\mathbb{R}^2 \times \mathbb{R}^2),\]
we have for all $u\in \mathscr{S}(\mathbb{R}^2)$, $\left(\mbox{Op} (a)\left(u\circ \chi^{-1}\right)\right)\circ \chi=\mbox{Op} (a^*)(u)$. Moreover  for $k\in \mathbb{N}$ there exists an increasing function $C_k$ such that for all $s \in \mathbb{R}$ and $(p,q)\in [1,+\infty]$
\begin{multline*}\left\Vert \mbox{Op} (a^*)(x,\xi)-\mbox{Op} \left(\sum_{\left\vert\alpha\right\vert <k}\frac{1}{\alpha!}\partial^\alpha a(\chi(x),([D\chi(x)]^{-1})^t\xi)Q_{\alpha}(\chi(x),\xi)\right)\right\Vert_{H^{s} \rightarrow H^{s-m+k}}\\ \leq C_k\left(\left\Vert D\chi\right\Vert_{L^\infty},\left\Vert D\chi^{-1}\right\Vert_{L^\infty},\left\Vert D\chi\right\Vert_{W^{(k-1)^+,\infty}}\right) M^m_{k+\lceil s\rceil} (a;k+2),\end{multline*}
where,
\[ Q_{\alpha}(x',\xi)=D^\alpha_{y'}(e^{i(\chi^{-1}(y')-\chi^{-1}(x')-D \chi^{-1}(x')(y'-x'))\cdot\xi})_{|y'=x'} \]
and $Q_{\alpha}$ is polynomial in $\xi$ of degree $\leq \frac{\left\vert \alpha\right\vert}{2}$, with $Q_{0}=1, Q_{1}=0$.
\end{itemize}

\end{theorem}

\subsection{Paradifferential calculus}
We start by the definition of symbols with limited spatial regularity.
\begin{definition}\label{def: para symbol}
Given $m \in \mathbb{R}$, $\Gamma^m_{\rho}(\mathbb{R}^2)$ denotes the space of locally bounded functions $a(x,\xi)$ on $\mathbb{R}^2\times (\mathbb{R}^2 \setminus 0)$, which are $C^\infty$ with respect to $\xi$ for $\xi \neq 0$ and such that, for all $\alpha \in \mathbb{N}^2$ and for all $\xi \neq 0$, the function $x \mapsto \partial^\alpha_\xi a(x,\xi)$ belongs to $W^{\rho,\infty}$ and there exists a constant $C_\alpha$ such that, for all $\epsilon>0$:
\begin{equation}\label{eq:def para} 
\forall \left\vert \xi\right\vert>\epsilon, \left\Vert\partial^\alpha_\xi a(\cdot,\xi)\right\Vert_{W^{\rho,\infty}}\leq C_{\alpha,\epsilon} (1+\left\vert \xi\right\vert)^{m-\left\vert \alpha\right\vert}. 
\end{equation}
The spaces $\Gamma^m_{\rho}(\mathbb{R}^2)$ are equipped with their natural Fr\'echet topology induced by the semi-norms defined by the best constants in \eqref{eq:def para} (see also \cite{metivier2008differential}):
\[M^m_{\rho}(a;n)=\sup_{\left\vert \alpha\right\vert\leq n} \ \sup_{\left\vert \xi\right\vert\geq\frac{1}{2}}\left\Vert(1+\left\vert \xi\right\vert)^{m-\left\vert \alpha\right\vert}\partial^\alpha_\xi a(.,\xi)\right\Vert_{W^{\rho,\infty}}, \text{ for } n\in \mathbb{N}.\]
\end{definition}

\begin{definition}
Define an admissible cutoff function as a function $\psi^{B,b}\in C^\infty(\mathbb{R}^4)$,  $B>1,b>0$ that verifies:
\begin{enumerate}
\item 
\[
\psi^{B,b}(\xi,\eta)=0 \text{ when }
\left\vert \xi\right\vert< B\left\vert \eta\right\vert+b,
\text{ and }
\psi^{B,b}(\xi,\eta)=1 \text{ when } \left\vert \xi\right\vert>B\left\vert \eta\right\vert+b+1.
\]
\item For all $(\alpha,\beta)\in \mathbb{N}^4,$ there exists $C_{\alpha_\beta}$, with $C_{0,0}\leq 1$, such that:
\begin{equation}\label{eq: grwth cut off para}
\forall(\xi,\eta): \left\vert \partial_\xi^\alpha \partial_\eta^\beta \psi^{B,b}(\xi,\eta)\right\vert\leq C_{\alpha,\beta} (1+\left\vert \xi\right\vert)^{-\left\vert \alpha\right\vert-\left\vert \beta\right\vert}.
\end{equation}
\end{enumerate}
\end{definition}

\label{def: para op}\begin{definition}
Consider a real numbers $m\in \mathbb{R}$, a symbol $a\in \Gamma^m_{\rho}(\mathbb{R}^2)$ and an admissible cutoff function $\psi^{B,b}$ define the paradifferential operator $T_a$ by
\[T_a u=\mbox{Op}\ (\sigma_a)u, \text{ where } \mathscr{F}_x(\sigma_a)(\xi,\eta)=\mathscr{F}_x(\sigma^{B,b}_a)(\xi,\eta)=\psi^{B,b}(\xi,\eta) \mathscr{F}_x (a)(\xi,\eta).\]
Let $G_{\psi^{B,b}}(x,\eta)=\mathscr{F}^{-1}_x\left(\psi^{B,b}(\cdot,\eta)\right)$ then $\sigma_a(\cdot,\eta)=G_{\psi^{B,b}}(\cdot,\eta)\ast a(\cdot,\eta)$.
\end{definition}

An important property of paradifferential operators is their action on functions with localised spectrum.
\begin{proposition}\label{prop: para act spctrm}
Consider two real numbers $m\in \mathbb{R}$, $\rho\geq 0$, a symbol $a\in \Gamma^m_0(\mathbb{R}^2)$, an admissible cutoff function $\psi^{B,b}$ and $u \in \mathscr{S}(\mathbb{R}^2)$.
\begin{itemize}
\item For $R>>b$, if $\mbox{supp}\  \mathscr{F} (u) \subset \{\left\vert \xi\right\vert\leq R\},$ then: 
\begin{equation}\label{eq:frq lc para bll}
\mbox{supp}\  \mathscr{F} \left(T_a u\right)\subset \left\{\left\vert \xi\right\vert\leq \left(1+\frac{1}{B}\right)R-\frac{b}{B}\right\},
\end{equation}
\item For $R>>b$, if $\mbox{supp}\  \mathscr{F} (u) \subset \left\{\left\vert \xi\right\vert\geq R\right\},$ then: 
\begin{equation}\label{eq:frq lc para rng}
\mbox{supp}\  \mathscr{F} \left(T_a u\right) \subset \left\{\left\vert \xi\right\vert\geq \left(1-\frac{1}{B}\right)R+\frac{b}{B}\right\},
\end{equation}
\end{itemize}
\end{proposition}

The main features of symbolic calculus for paradifferential operators are given by the following theorems taken from \cite{metivier2008differential}, \cite{said2023paracomposition} and \cite{taylor2000tools}.
\begin{theorem}\label{thm: para continuity}
Consider $m \in \mathbb{R}$ then for $a\in \Gamma^m_0(\mathbb{R}^2)$, then $T_a$ is of order $m$, more precisely for $s \in \mathbb{R}$ there exists a constant $C_s$ such that:
\[\left\Vert T_a \right\Vert_{H^{s} \rightarrow H^{s-m}}\leq C_s M^m_0(a,2).\]
\end{theorem}
\begin{theorem} \label{thm: symb calc para} 
Let $m,m' \in \mathbb{R}$, and $\rho>0$, $a \in \Gamma^m_\rho(\mathbb{R}^2)$and $b \in \Gamma^{m'}_\rho(\mathbb{R}^2)$. 
\begin{itemize}

\item Composition: Then $T_a T_b$ is a paradifferential operator with symbol: $$a \otimes b\in \Gamma^{m+m'}_\rho(\mathbb{R}^2),\text{ more precisely,}$$
\[
T^{\psi^{B,b}}_a T^{\psi^{B',b}}_b= T^{\psi^{\frac{BB'}{B+B'+1},b}}_{a\otimes b}.
\]
Moreover $T_a T_b- T_{a\#b}$ is of order $m+m'-\rho$ where $a \#b $ is defined by:
\[a \#b=\sum_{\left\vert \alpha\right\vert<\rho }\frac{1}{i^{\left\vert \alpha\right\vert}\alpha!} \partial^\alpha_\xi a \partial^\alpha_x b, \]
and there exists $r\in \Gamma^{m+m'-\rho}_0(\mathbb{R})$ such that:
\[ M^{m+m'-\rho}_0(r) \leq C_\rho (M^m_\rho (a;2) M^{m'}_0(b;\lceil \rho \rceil+2)+M^m_0 (a;\lceil \rho \rceil+2) M^{m'}_\rho(b;2)), \]
and we have
 \[T^{\psi^{B,b}}_a T^{\psi^{B',b}}_b- T^{\psi^{\frac{BB'}{B+B'+1},b}}_{a\#b}=T^{\psi^{\frac{BB'}{B+B'+1},b}}_r. \]

\item  Adjoint: The adjoint operator of $T_a$, $T_a^t$ is a paradifferential operator of order m with  symbol $a^t$ defined by:
\[a^t=\sum_{\left\vert \alpha\right\vert<\rho} \frac{1}{i^{\left\vert \alpha\right\vert}\alpha!}\partial^\alpha_\xi \partial^\alpha_x \bar{a}. \]
Moreover, for all $s \in \mathbb{R}$  there exists a constant $C_{s,\rho}$ such that
\[ \left\Vert T_a^t-T_{a^t}\right\Vert_{H^s\rightarrow H^{s-m+\rho}} \leq C_{s,\rho} M^m_\rho (a). \]
\end{itemize}
\end{theorem}
If $a=a(x)$ is a function of $x$ only then the paradifferential operator $T_a$ is called a paraproduct. 
It follows from Theorem \ref{thm: symb calc para} and the Sobolev embedding that:
\begin{itemize}
\item If $a \in H^\alpha(\mathbb{R}^2)$ and $b \in H^\beta(\mathbb{R}^2)$ with $\alpha,\beta>1$, then
\[T_aT_b-T_{ab} \text{ is of order } -\bigg( \min\{\alpha,\beta\}-1 \bigg).\]
\item If $a \in H^\alpha(\mathbb{R}^2)$ with $\alpha>1$, then
\[T_a^*-T_{\bar{a}} \text{ is of order } -\bigg(\alpha-1 \bigg).\]
\item If $a \in W^{r,\infty}(\mathbb{R}^2)$ and $r\geq 0$  then:
\[\left\Vert au-T_au\right\Vert_{H^r(\mathbb{R}^2)} \leq C \left\Vert a \right\Vert_{W^{r,\infty}} \left\Vert u\right\Vert_{L^2}.\]
\end{itemize}
An important feature of paraproducts is that they are well defined for function $a=a(x)$ which are not $L^\infty$ but merely in some Sobolev space $H^{1-m}$, see \cite{bahouri2011fourier}.
\begin{proposition}
Let $m>0$. If $a\in H^{1-m}(\mathbb{R}^2)$ and $u \in H^\mu(\mathbb{R})$ then $T_au \in  H^{\mu-m}(\mathbb{R})$. Moreover there exists $C_\mu$ such that
\[ \left\Vert T_a u\right\Vert_{H^{\mu -m}}\leq C_\mu \left\Vert a\right\Vert_{H^{1 -m}}\left\Vert u\right\Vert_{H^{\mu}}. \]
\end{proposition}

A main feature of paraproducts is the existence of paralinearisation theorems which allow us to replace nonlinear expressions by paradifferential expressions, at the price of error terms which are smoother than the main terms.

\begin{theorem}[From \cite{bahouri2011fourier}] \label{thm: para product} Let $\alpha, \beta,\kappa \in \mathbb{R} $ be such that $\alpha,\beta> 1$ and $\kappa\geq 0$, then
\item Bony's Linearization Theorem: For all $C^\infty$ function F, if $a \in H^\alpha (\mathbb{R})$ then
\[ F(a)- F(0)-T_{F'(a)}a \in H^{2\alpha-1} (\mathbb{R}^2). \]
\item If $a\in H^\alpha(\mathbb{R}^2)$, $b\in H^\beta(\mathbb{R}^2)$ and $c\in W^{\kappa,\infty}(\mathbb{R}^2)$, then $R(a,b)=ab-T_ab-T_ba \in H^{\alpha+ \beta-1} (\mathbb{R}^2)$ and $R(a,c)=ac-T_ac-T_ca \in H^{\alpha+ \kappa} (\mathbb{R}^2)$. Moreover there exists a positive constant $C$ independent of $a$, $b$ and $c$ such that:
\begin{equation}\label{eq:resid para-product}
\begin{cases}
\left\Vert R(a,b)\right\Vert_{H^{\alpha+ \beta-1} }=\left\Vert ab-T_ab-T_ba\right\Vert_{H^{\alpha+ \beta-1} }\leq C  \left\Vert a\right\Vert_{H^\alpha} \left\Vert b\right\Vert_{H^\beta},\\
\left\Vert R(a,c)\right\Vert_{H^{\alpha+ \kappa} }\leq C  \left\Vert a\right\Vert_{H^\alpha} \left\Vert c\right\Vert_{W^{\kappa,\infty}}.
\end{cases}
\end{equation}
The residual term verifies for $\lambda>>b$, if $\mbox{supp}\  \mathscr{F} (a) \subset \left\{\left\vert \xi\right\vert\geq \lambda\right\}$ and  $\mbox{supp}\  \mathscr{F} (b) \subset \left\{\left\vert \xi\right\vert\geq \lambda\right\}$then: 
\begin{equation}\label{eq:frq lc resid para rng}
\mbox{supp}\  \mathscr{F} \left(R(a,b)\right) \subset \left\{\left\vert \xi\right\vert\geq \left(1-\frac{1}{B}\right)\lambda+\frac{b}{B}\right\}.
\end{equation}
\end{theorem}

\subsection{Paracomposition}\label{sec: paracomp}
We recall the main properties of the paracomposition operator first introduced by S. Alinhac in \cite{alinhac1986paracomposition} to treat low regularity change of variables. Here we present the results we reviewed and generalised in some cases in \cite{said2023paracomposition,taylor2000tools}.
\begin{theorem} \label{thm: def para comp}
 Let $\chi:\mathbb{R}^2 \rightarrow \mathbb{R}^2$ be a $C^1\left(\mathbb{R}^2\right)$ diffeomorphism with $D\chi \in W^{r,\infty}$, $r>0, r\notin \mathbb{N}$ and take $s \in \mathbb{R}$ then the following map is continuous:  
   \begin{align*}
   H^s(\mathbb{R}^2) &\rightarrow H^s(\mathbb{R}^2)\\
  u &\mapsto \chi^* u=\sum_{k\geq 0}  \sum_{\substack{l\geq 0 \\ k-N \leq l \leq k+N}}P_l(D)u_k\circ \chi,
\end{align*} 
where $N \in \mathbb{N}^*$ is such that $2^{N}>\sup_{k,\mathbb{R}^d} \left\vert \Phi_k D\chi\right\vert^{-1}$ and $2^{N}>\sup_{k,\mathbb{R}^2} \left\vert \Phi_k D\chi\right\vert$. For two distinct choices of $N$ and $\hat{N}$ defining two operators $\chi^*$ and $\hat{\chi}^*$ we have 
\[
\left\Vert \chi^*u- \hat{\chi}^*u\right\Vert_{H^{s+r}} \leq C_{s+r}(\left\Vert D\chi\right\Vert_{L^\infty},\left\Vert D\chi^{-1}\right\Vert_{L^\infty})\left\Vert D\chi \right\Vert_{W^{r,\infty}}\left\Vert u \right\Vert_{H^s}.
\]

Taking $\tilde{\chi}:\mathbb{R}^d \rightarrow \mathbb{R}^d$ a $C^{1,\tilde{r}}\left(\mathbb{R}^2\right)$ diffeomorphism with $D\chi \in W^{\tilde{r},\infty}$ map with $\tilde{r}>0, r\notin \mathbb{N}$, then the previous operation has the natural fonctorial property:
\[\forall u \in H^s(\mathbb{R}^2) , \chi^* \tilde{\chi}^* u= ({\chi \circ \tilde{\chi}})^* u +Ru,\]   
 \[\text{with, }  R:H^s(\mathbb{R}^2) \rightarrow H^{s+min(r,\tilde{r})}(\mathbb{R}^2) \text{ continuous}.\]
  \end{theorem}
We now give the key paralinearisation theorem taking into account the paracomposition operator. 
\begin{theorem}  \label{thm: lin paracomp}
 Consider $s\in \mathbb{R}$, $u\in H^{s}(\mathbb{R}^2)$ and let $\chi:\mathbb{R}^2 \rightarrow \mathbb{R}^2$ be a $C^{1}\left(\mathbb{R}^2\right)$  diffeomorphism with $D\chi \in W^{r,\infty}(\mathbb{R}^2)$, $r>0, r\notin \mathbb{N}$. Then:
 \[u \circ \chi(x)=\chi^* u(x)+ T_{Du\circ \chi}\cdot\chi(x)+ R(x)\]
 where the paracomposition given in the previous theorem verifies the estimates:
 \[ \left\Vert \chi^* u\right\Vert_{H^s}\leq C_s(\left\Vert D\chi\right\Vert_{L^\infty},\left\Vert D\chi^{-1}\right\Vert_{L^\infty})\left\Vert u\right\Vert_{H^s},\]
 \[Du\circ \chi \in  \Gamma^0_{0}(\mathbb{R}^2) \text{ for $u$  Lipschitz,}   \]
and the remainders verify the estimate if $r+s>0$ 
\[  \left\Vert R\right\Vert_{H^{s+r}} \leq C_{s+r}(\left\Vert D\chi\right\Vert_{W^{r,\infty}},\left\Vert D\chi^{-1}\right\Vert_{L^\infty})\left\Vert u\right\Vert_{H^{s}}. \]	
	Finally the commutation between a paradifferential operator $a \in \Gamma^m_{\beta}(\mathbb{R}^d)$ and a paracomposition operator $\chi^*$ is given by the following
\[
	 \chi^* T_a u =T_{a^*} \chi^* u+T_{{q}^*} \chi^* u  \text{ with } q \in \Gamma^{m-\min(r,\beta)}_{0}(\mathbb{R}^d),
\]
where $a^*$ has the asymptotic expansion
\begin{equation}\label{eq:paracom exp}
a^*(x,\xi)=  \sum_{\substack{\left\vert \alpha\right\vert\leq \lfloor min(r,\beta) \rfloor}}\frac{1}{\alpha!}\partial^\alpha a(\chi(x),([D\chi(x)]^{-1})^t \xi)Q_{\alpha}(\chi(x),\xi),
\end{equation}
where,
\[ Q_{\alpha}(x',\xi)=D^\alpha_{y'}(e^{i(\chi^{-1}(y')-\chi^{-1}(x')-D \chi^{-1}(x')(y'-x'))\cdot\xi})_{|y'=x'}, \]
and $Q_{\alpha}$ is polynomial in $\xi$ of degree $\leq \frac{\left\vert \alpha\right\vert}{2}$, with $Q_{0}=1, Q_{1}=0$.
\end{theorem} 
\begin{remark}
    A close inspection of the proofs in \cite{alinhac1986paracomposition,taylor2000tools,said2023paracomposition} shows that the increasing functions $C_s$ and $C_{s+r}$ increase at most polynomially fast.
\end{remark}

\end{document}